\documentclass{amsart}
\usepackage{ amsmath, amsthm, amsfonts, hyperref, graphicx, ifpdf}
\usepackage[dvipsnames,usenames]{color}
\hypersetup{
   bookmarks=true,         
   unicode=false,          
   pdftoolbar=true,        
   pdfmenubar=true,        
   pdffitwindow=false,     
   pdfstartview={},    
   pdftitle={},    
   pdfauthor={Longzhi Lin},     
   pdfsubject={},   
   pdfcreator={Longzhi Lin},   
   pdfproducer={Longzhi Lin}, 
   pdfkeywords={}, 
   pdfnewwindow=true,      
   colorlinks=true,       
   linkcolor=blue,          
   citecolor=blue,        
   filecolor=magenta,      
   urlcolor=MidnightBlue          
}

\newtheorem{theorem}{Theorem}[section]
\newtheorem{lemma}[theorem]{Lemma}
\newtheorem{proposition}[theorem]{Proposition}
\newtheorem{corollary}[theorem]{Corollary}

\theoremstyle{definition}
\newtheorem{definition}[theorem]{Definition}

\theoremstyle{remark}
\newtheorem{remark}[theorem]{Remark}
\newtheorem{notation}[theorem]{\bf{Notations}}

\numberwithin{equation}{section}

\begin{document}

\title[Uniformity of harmonic map heat flow at infinite time] {Uniformity of harmonic map heat flow at infinite time}
\author{Longzhi Lin}
\address[L.~Lin]{Department of Mathematics\\Rutgers University\\110 Frelinghuysen Road\\Piscataway, NJ 08854-8019\\USA}
\email{lzlin@math.rutgers.edu}
\date{\today}

\begin{abstract}
We show an energy convexity along any harmonic map heat flow with small initial energy and fixed boundary data on the unit $2$-disk. In particular, this gives an affirmative answer to a question raised by W. Minicozzi asking whether such harmonic map heat flow converges uniformly in time strongly in the $W^{1,2}$-topology, as time goes to infinity, to the unique limiting harmonic map.

\end{abstract}

\maketitle {}

\section{Introduction}

Given a compact Riemannian manifold $\mathcal{M}$ and a closed (i.e., compact and without boundary) Riemannian manifold $\mathcal{N}$ which is an isometrically embedded submanifold of
$\mathbb{R}^n$, we can define the \textit{Dirichlet} energy of a map $u\in W^{1,2}(\mathcal{M}, \mathcal{N})$:
\begin{equation}\label{energy} \text{Energy}(u)\,=\,E(u)\,=\,\frac{1}{2}\,\int_{\mathcal{M}}\,|\nabla u|^2 dv_{\mathcal{M}}\,,
\end{equation}
where $W^{1,2}(\mathcal{M}, \mathcal{N})$ is the class of maps
$$\left\{ u \in L^{1}_{loc}(\mathcal{M},\mathbb{R}^n): \,\,\int_{\mathcal{M}}|\nabla u|^2 dv_{\mathcal{M}} < +\infty, \,\,u(x)\in \mathcal{N} \,\text{ a.e. } x \in \mathcal{M}\right\}.$$
The tension field $\tau(u)\in \Gamma(u^*(T\mathcal{N}))$ is the
vector field along $u$ representing the negative $L^2$-gradient of
$E(u)$\,. A weakly harmonic map $u$ from $\mathcal{M}$ to $\mathcal{N}$ is a critical point of the energy functional $E(u)$ in the distribution sense, i.e., the tension field $\tau(u)$ vanishes, and it solves the Euler-Lagrange equation
\begin{equation}\label{HM}
-\Delta_{\mathcal{M}} u\,=\, A(u)(\nabla u, \nabla u)\,,
\end{equation}
where $u=(u^1,...,u^n)$ and $A(u)$ denotes the second fundamental form of $\mathcal{N} \hookrightarrow \mathbb{R}^n$ at the point $u$. We refer to this system of elliptic equations as the harmonic map equation.

 A natural way to control the tension field for an energy minimizing sequence of maps and to get the existence of harmonic maps from $\mathcal{M}$ to $\mathcal{N}$ is to
consider the initial(-boundary) value problem:
\begin{equation} \label{HHF1}
  \left\{
   \begin{aligned}
    & u_t-\Delta_{\mathcal{M}} u = A(u)(\nabla u, \nabla u), && \text{on } \mathcal{M}\times (0,T)  \\
   & u(x,0) = u_0(x), && \text{for }x\in \mathcal{M}    \\
    &u(x,t) = \chi(x) = u_0|_{\partial \mathcal{M}}, &&  \text{for all } t\geq 0,\, x\in \partial \mathcal{M}  \,\,\text{  if } \,\partial \mathcal{M} \neq \emptyset\,,
   \end{aligned}
 \right.
   \end{equation}
where $u=(u^1,...,u^n)$ and $T>0$. We refer to this system of parabolic equations as the harmonic map heat flow, to the map $u_0$ as the initial data, and to the map $\chi$ as the boundary data. Given $u_0 \in W^{1,2}(\mathcal{M},\mathcal{N})$ and $\chi = u_0|_{\partial \mathcal{M}} \in W^{\frac{1}{2},2}(\partial \mathcal{M},\mathcal{N})$, we define $u \in W^{1,2}(\mathcal{M}\times [0,T], \mathcal{N})$ to be the weak solution of \eqref{HHF1} if
\begin{equation}
\int_0^T \int_{\mathcal{M}} \langle u_t, \xi\rangle + \langle \nabla u, \nabla\xi\rangle - \langle  A(u)(\nabla u, \nabla u), \xi\rangle dxdt = 0
\end{equation}
for any $\xi \in C^{\infty}_c(\mathcal{M}\times (0,T), \mathbb{R}^n)$.

In the fundamental paper \cite{ES} where the harmonic map heat flow was first introduced, Eells and Sampson proved that the harmonic map heat flow exists for all time in the case that the source domain $\mathcal{M}$ (of arbitrary dimensions) is without boundary and the target manifold $\mathcal{N}$ has non-positive sectional curvature. They also proved that there exists some sequence of times $t_i \nearrow +\infty$ such that
$$
u_\infty = \lim_{i\to\infty}u (\cdot,t_i)
$$
is a harmonic map from $\mathcal{M}$  to $\mathcal{N}$. The case in which the source domain $\mathcal{M}$ has boundary was dealt with by Hamilton in \cite{Ha} under the same curvature assumption on $\mathcal{N}$. The question of uniformity of the convergence in time of the flow considered by Eells and Sampson in \cite{ES} was left open at that stage, but it was settled later by Hartman in \cite{Har}. We shall state their results in the following theorem.

\begin{theorem}[\cite{ES, Har}]\label{ESHar}
Suppose that $\mathcal{M}$ and $\mathcal{N}$ are two closed Riemannian manifolds and that $\mathcal{N}$ has non-positive sectional curvature. Then given any $u_0 \in C^1(\mathcal{M}, \mathcal{N})$, the harmonic map heat flow has a unique solution $u\in C^1(\mathcal{M} \times [0,\infty), \mathcal{N}) \cap C^\infty(\mathcal{M} \times (0,\infty), \mathcal{N})$. Moreover,
\begin{equation}
u_\infty = \lim_{t\to\infty}u (\cdot,t)
\end{equation}
exists uniformly in $C^k$-topology for all $k\geq 0$ and $u_\infty$ is a harmonic map homotopic to $u_0$.
\end{theorem}

Other similar uniformity results were obtainable under various of assumptions on the target manifold $\mathcal{N}$, such as $\mathcal{N}$ is real analytic (see Simon \cite{S}) or $\mathcal{N}$ admits a strictly convex function, see also Topping's interesting work \cite{T1} in this direction for harmonic map heat flow in a special case in which both the source and target manifolds are $2$-spheres $\mathbb{S}^2$.

When the dimension of the source domain $\mathcal{M}$ is two it is particularly interesting because the energy functional $E(u)$ and the harmonic map equation \eqref{HM} are conformally invariant in this critical dimension. Regarding the harmonic map heat flow \eqref{HHF1} from surfaces to a general closed target manifold $\mathcal{N}$, the first fundamental work was for the case $\partial \mathcal{M}=\emptyset$ which was due to Struwe \cite{St1}, where ``bubbles'' may
occur and have been analyzed in detail.  This result was then extended to the case $\partial \mathcal{M}\neq\emptyset$ with Dirichlet boundary condition by Chang in \cite{Ch}. If the initial energy $E(u_0)$ is sufficiently small, it is well-known by now that the weak solution of \eqref{HHF1} is smooth (in the interior) by the results of Freire \cite{Fr1, Fr2} using the so called moving frame technique introduced by H\'elein (see e.g. \cite{He1}). We state their $\varepsilon$-regularity theorem that we shall require in this paper in the following. For the self-containedness of the paper, we will include an alternative proof of this $\varepsilon$-regularity theorem using the main tool of our current work that we call Rivi\`ere's gauge decomposition, see Theorem \ref{Regularity}.

\begin{theorem}[\cite{Fr1,Fr2}, cf. \cite{St1, Ch, Wa}] \label{Chang} Let $\mathcal{M}$ be a simply-connected compact Riemannian surface and $\mathcal{N}$ be a closed Riemannian manifold. There exists $\varepsilon_0>0$ depending only on $\mathcal{M}$ and $\mathcal{N}$ such that the following is true. For each initial data $u_0 \in W^{1,2}(\mathcal{M}, \mathcal{N})$ with $E(u_0) < \varepsilon_0$ and the boundary data $\chi = u_0|_{\partial \mathcal{M}}$ in the case that $\partial \mathcal{M} \neq \emptyset$, there exists a unique global weak solution $u \in W^{1,2}(\mathcal{M}\times [0,\infty), \mathcal{N})$ for which $E(u(\cdot, t))$ is non-increasing in $t$. Also, $u$ is smooth in $\mathcal{M}\times [1,\infty)$ and for any $t_2>t_1\geq 1$ we have
\begin{equation}\label{EnergyDrease}
2\int_{t_1}^{t_2} \int_{\mathcal{M}}|u_t|^2 \,=\,  \int_{B_1} |\nabla u(\cdot, t_1)|^2-\int_{B_1}|\nabla u(\cdot, t_2)|^2\,.
\end{equation}
Moreover, there exists some sequence of times $t_i \nearrow +\infty$ such that
\begin{equation}\label{ChangConv}
u_\infty = \lim_{i\to\infty}u (\cdot,t_i)
\end{equation}
exists in the $C^k$-topology for any $k\geq0$ and $u_\infty$ is a harmonic map from $\mathcal{M}$ to $\mathcal{N}$.
\end{theorem}

\begin{remark}
In particular, in order to avoid the ``bubble'' (singularity) along the harmonic map heat flow, \textit{a priori} we may choose $\varepsilon_0 < K_1 + K_2$ where
$$
K_1\,=\,\inf\,\{E(v)\,\big|\,v\in W^{1,2}(\mathcal{M},\mathcal{N})\text{ and } v|_{\partial \mathcal{M}}=\chi\}
$$
and
$$
K_2\,=\,\inf\,\{E(v)\,\big|\,v:\,\mathbb{S}^2\to \mathcal{N} \text{ is non-constant and harmonic}\} \,>0\,.
$$
\end{remark}

\begin{remark}
Freire's regularity results for harmonic map heat flow in \cite{Fr1,Fr2} is a parabolic version of the regularity theorem of H\'elein which states that weakly harmonic maps from surfaces are regular (see e.g. \cite{He1}).
\end{remark}

A tempting question to ask is that, for a general closed target manifold $\mathcal{N}$ (without imposing additional geometric assumptions on $\mathcal{N}$), whether one could establish uniformity results for the harmonic map heat flow similar to Theorem \ref{ESHar}. In particular, is the convergence \eqref{ChangConv} in Theorem \ref{Chang} uniform for all time in the natural $W^{1,2}$-topology, say? In view of the conformal-invariance of the energy functional $E(u)$ in dimension two, the condition of small energy seems to be a natural candidate to work with in order to get such uniformity of the convergence in time for the flow. We will show in the following that this is indeed the case. In what follows we will concentrate on the case that source domain $\mathcal{M}$ is a simply-connected compact Riemannian surface with boundary. More precisely we focus on domains which are conformally equivalent to the unit $2$-disk $B_1 \subset \mathbb{R}^2$. From now on we will only work on $B_1$:
\begin{equation} \label{HHF}
  \left\{
   \begin{aligned}
    & u_t-\Delta u = A(u)(\nabla u, \nabla u), && \text{on } B_1\times (0,T)  \\
   & u(x,0) = u_0(x), && \text{for }x\in B_1    \\
    &u(x,t) = \chi(x) = u_0|_{\partial B_1}, &&  \text{for all } t\geq 0 \,\text{ and }\, x\in \partial B_1 \,,
   \end{aligned}
 \right.
   \end{equation}
where $\Delta$ is the usual Laplacian $\Delta = \sum_{i=1}^2\frac{\partial^2}{\partial x_i^2}$ in $\mathbb{R}^2$. All the arguments could be easily modified to apply to the general case.

\begin{notation}\label{Del-perp}
In what follows, $\nabla = (\partial_{x}, \partial_{y})$ is the gradient operator in $\mathbb{R}^2$ and $\nabla^{\perp} = (-\partial_{y}, \partial_{x})$ denotes the orthogonal gradient (i.e., $\nabla^{\perp}$ is the $\nabla$-operator rotated by $\pi/2$).
\end{notation}
Now we state the main theorem of this paper.

\begin{theorem}\label{MainThm}
Let $\mathcal{N}$ be a closed Riemannian manifold. There exist $\varepsilon_0, T_0 >0$ depending only on $\mathcal{N}$ such that if $u \in W^{1,2}(B_1\times[0,\infty),\mathcal{N})$ is a global weak solution to the harmonic map heat flow \eqref{HHF} with $E(u_0) < \varepsilon_0$, $E(u(\cdot, t))$ is non-increasing in $t$ and $u(\cdot, t)|_{\partial B_1} = \chi$ for all $t\geq 0$, then for all $t_2 >t_1 \geq T_0$ we have the energy convexity
\begin{equation}\label{MainConv}
\frac{1}{4}\int_{B_1} |\nabla u(\cdot, t_1)-\nabla u(\cdot, t_2)|^2\,\leq \, \int_{B_1} |\nabla u(\cdot, t_1)|^2-\int_{B_1}|\nabla u(\cdot, t_2)|^2\,.
\end{equation}
\end{theorem}

\begin{remark}
We do not know if the energy convexity \eqref{MainConv} holds for all $t_2>t_1\geq 0$, cf. \cite{Wa}. In the following arguments we agree to let $\varepsilon_0$ be sufficiently small and $T_0$ be sufficiently large as needed.
\end{remark}

Our approach to the proof of Theorem \ref{MainThm} is based on the technique that we call Rivi\`ere's gauge decomposition introduced by Rivi\`ere in \cite{Riv1}, see also Section \ref{Sec2}. Immediate applications of Theorem \ref{MainThm} are the following corollaries.

\begin{corollary}\label{COR}
Let $\mathcal{N}$ be a closed Riemannian manifold. There exists $\varepsilon_0>0$ depending only on $\mathcal{N}$ such that if $u \in W^{1,2}(B_1\times[0,\infty),\mathcal{N})$ is a global weak solution to the harmonic map heat flow \eqref{HHF} with $E(u_0) < \varepsilon_0$, $E(u(\cdot, t))$ is non-increasing in $t$ and $u(\cdot, t)|_{\partial B_1} = \chi$ for all $t\geq 0$, then
\begin{equation} u(\cdot,t) \to u_{\infty} \text{ uniformly as } t \to +\infty \text{ strongly in } W^{1,2}(B_1, \mathbb{R}^n)\,,
\end{equation}
where $u_{\infty}$ is the unique harmonic map with $E(u_{\infty})<\varepsilon_0$ and boundary data $\chi$.
\end{corollary}

\begin{corollary}
Let $\mathcal{M}$ be a two dimensional domain that is conformally equivalent to $B_1$ and has smooth boundary, and let $\mathcal{N}$ be a closed Riemannian manifold. Suppose the initial energy $E(u_0) < \varepsilon_0$, then the harmonic map heat flow \eqref{HHF1} with initial data $u_0\in C^{2,\alpha}(\overline{\mathcal{M}}, \mathcal{N})$ and boundary data $\chi\in C^{2,\alpha}(\partial \mathcal{M}, \mathcal{N})$, considered by Chang in \cite{Ch}, converges uniformly in time strongly in $W^{1,2}(\mathcal{M}, \mathcal{N})$ to the unique harmonic map $u_\infty \in C_{\chi}^{2,\alpha}(\overline{\mathcal{M}}, \mathcal{N})$.
\end{corollary}

\begin{remark}
We do not know if a harmonic map heat flow can be non-uniform without the small energy assumption. In view of the non-uniqueness results of Brezis and Coron \cite{BC1} and Jost \cite{Jo} for harmonic maps (with large energy) sharing the same boundary data on $\partial B_1$, it is quite possible that the small energy assumption is necessary for the energy convexity and uniform convergence of the flow in Theorem \ref{MainThm} and Corollary \ref{COR} to hold.
\end{remark}

\begin{remark}
In \cite{CM1} Colding and Minicozzi showed an energy convexity for weakly harmonic maps with small energy on $B_1$: there exists $\varepsilon_0>0$ such that if $u, v \in W^{1,2}(B_1, \mathcal{N})$ with $u|_{\partial B_1}=v|_{\partial B_1}, E(u)<\varepsilon_0$, and $u$ is weakly harmonic, then we have the energy convexity
\begin{equation}\label{TEMP49}
\frac{1}{2}\int_{B_1} |\nabla v-\nabla u|^2\,\leq\, \int_{B_1} |\nabla v|^2-\int_{B_1}|\nabla u|^2\,.
\end{equation}
See our recent work \cite{LL} for an alternative proof of this energy convexity using the same techniques used in the present paper. A direct consequence of \eqref{TEMP49} is that $u_\infty$ in Corollary \ref{COR} is unique in the class
$$\{v\in W^{1,2}(B_1, \mathbb{R}^n): \, E(v) < \varepsilon_0 \text{ and }v|_{\partial B_1}= \chi\}\,,$$
see \cite[Corollary 3.3]{CM1}.
\end{remark}

The paper is organized as follows. In Section \ref{Sec1} we present some heuristic arguments and elaborate on the idea of the proof of the main Theorem \ref{MainThm}. In Section \ref{Sec2} we review the main tool of our proof, namely, Rivi\`ere's gauge decomposition technique adapted to the case of harmonic map heat flow{\footnote{Alternatively, one may use the results of Freire \cite{Fr1,Fr2} where H\'elein's estimates involving moving frame were used, however, Rivi\`ere's approach seems to be more elementary and self-contained.}}. In Section \ref{Sec3} we show improved estimates for Rivi\`ere's  matrices $B$ and $P$, which are the two key ingredients of our proof. We finish the proof of our main theorem in Section \ref{Sec4}.

\subsection*{Acknowledgement}
The author would like to thank Professor William Minicozzi for his continued guidance and support during the author's Ph.D. study at The Johns Hopkins University. The author would also like to thank his collaborator Tobias Lamm for all the valuable discussions.

\section{Heuristic arguments and the idea of the proof}\label{Sec1}
In this section we will present some heuristic arguments and sketch the basic idea of the proof of Theorem \ref{MainThm}. We will abbreviate $u(\cdot,t)$ to $u(t)$. In order to prove the energy convexity \eqref{MainConv} along the harmonic map heat flow, i.e., there exists some $T_0>0$ such that for all $t_2> t_1 \geq T_0$ we have
\begin{equation}
\frac{1}{4}\int_{B_1} |\nabla u(\cdot, t_1)-\nabla u(\cdot, t_2)|^2\,\leq \, \int_{B_1} |\nabla u(\cdot, t_1)|^2-\int_{B_1}|\nabla u(\cdot, t_2)|^2\,,
\end{equation}
it suffices to show
\begin{equation}\label{Convexity1}
\Psi\,\geq\, -\left(\int_{B_1} |\nabla u(t_1)|^2 -\int_{B_1}|\nabla u(t_2)|^2\right) -\frac{1}{2}\int_{B_1} |\nabla u(t_1)-\nabla u(t_2)|^2\,,
\end{equation}
where (using that $u(\cdot,t)|_{\partial B_1} = \chi$ for all $t\geq0$ and the flow equation \eqref{HHF})
\begin{align}
 \Psi := &\int_{B_1} |\nabla u(t_1)|^2-\int_{B_1}|\nabla u(t_2)|^2-\int_{B_1} |\nabla u(t_1)-\nabla u(t_2)|^2 \notag\\
 = &\,2\int_{B_1} \langle \nabla u(t_1) - \nabla u(t_2) ,\nabla u(t_2)\rangle \notag \\
  = &\,-2\int_{B_1} \langle u(t_1) - u(t_2) , u_t(t_2)- A(u)(\nabla u, \nabla u)(t_2)\rangle\,.\label{Convexity2}
\end{align}
Now note that for any $p, q \in \mathcal{N}$, there exists some constant $C>0$ depending only on $\mathcal{N}$ such that $\left|(p-q)^{\perp}\right|\leq C |p-q|^2$, where the superscript $\perp$ denotes the normal component of a vector (see e.g. \cite[Lemma A.1]{CM2}). Therefore, using the fact that $A(u)(\nabla u, \nabla u) \perp T_u \mathcal{N}$ and Cauchy-Schwarz inequality, \eqref{Convexity2} yields
\begin{align}
\Psi\geq & -2 \left( \int_{B_1} |u(t_1)-u(t_2)|^2\right)^{\frac{1}{2}} \left( \int_{B_1} |u_t(t_2)|^2\right)^{\frac{1}{2}} - C\int_{B_1} |(u(t_1) - u(t_2))^\perp||\nabla u(t_2)|^2\notag\\
\geq & -2 \sqrt{t_2-t_1}\left( \int_{t_1}^{t_2}\int_{B_1} |u_t|^2\right)^{\frac{1}{2}} \left( \int_{B_1} |u_t(t_2)|^2\right)^{\frac{1}{2}} - C\int_{B_1} |u(t_1) - u(t_2)|^2|\nabla u(t_2)|^2,\notag
\end{align}
where we also used the smoothness and compactness of the target manifold $\mathcal{N}$. Here and throughout the rest of the paper, $C>0$ will denote a universal constant depending only on $\mathcal{N}$ unless otherwise stated.

Since we have \eqref{EnergyDrease} and $\varepsilon_0$ can always be chosen sufficiently small, we know that \eqref{Convexity1} will be achieved if we can show the following two key propositions.

\begin{proposition}\label{utEst} Let $u(x,t)$ be as in Theorem \ref{MainThm}, then there exists $T_0>0$ such that for all $t_2 > t_1 \geq T_0$ we have
\begin{equation}\label{Convexity3}
\int_{B_1} |u_t(t_2)|^2\,\leq\, \frac{1}{t_2 - t_1}\int_{t_1}^{t_2}\int_{B_1} |u_t|^2\,.
\end{equation}
\end{proposition}
\begin{remark}\label{Struwe1}
The key point of Proposition \ref{utEst} is that \eqref{Convexity3} is valid for all $t_2 > t_1 \geq T_0$. We will see that, in fact, $\int_{B_1} |u_t(t)|^2$ is non-increasing along the flow after $T_0$, which yields \eqref{Convexity3}, cf. Lemma \ref{utEst1} and \eqref{utnon} below. A similar but weaker estimate was shown in \cite{St1} when the source domain of the heat flow is boundaryless (see equation (3.5) of \cite{St1}), which turned out to be the key estimate needed in Struwe's proof.
\end{remark}

\begin{proposition}\label{GradEst}Let $u(x,t)$ be as in Theorem \ref{MainThm}, then there exists $T_0>0$ such that for all $t_2 > t_1 \geq T_0$ we have
\begin{equation}\label{Convexity4}
\int_{B_1}|u(t_1) - u(t_2)|^2|\nabla u(t_2)|^2\,\leq\, C \varepsilon_0\int_{B_1}|\nabla u(t_1)-\nabla u(t_2)|^2\,.
\end{equation}
\end{proposition}
If one was able to get
\begin{equation}\label{TEMP50}
\|\nabla u(t_2)\|_{L^\infty(B_1)}\leq C\sqrt{\varepsilon_0}\,,
\end{equation}
then \eqref{Convexity4} would have been automatically true by Poincar\'e's inequality. However, without imposing any regularity information on the boundary data $\chi$, it will be hopeless to get such strong global pointwise gradient estimate. In fact, even if we look at the stationary case, i.e., $W^{1,2}$ -weakly harmonic maps on $B_1$, it is easy to convince oneself that it is unreasonable to expect regularity with global estimates on the whole $B_1$ better than $W^{2,2}$ in general.

Nevertheless, not all hope is lost to show estimates \eqref{Convexity3} and \eqref{Convexity4}. Indeed, the following lemma is true which validates Proposition \ref{GradEst} under some extra assumption.

\begin{lemma}\label{GradEst1}
Let $u(x,t)$ be as in Theorem \ref{MainThm} and suppose that for all $t_2 > t_1 \geq T_0\geq 1$ we can solve the following Dirichlet problem for $\psi \in W^{1,2}_0 \cap L^\infty(B_1)$:
\begin{equation}\label{PSI1}
\left\{
   \begin{aligned}
     \Delta \psi\,& = \,|\nabla u(t_2)|^2 && \text{in }\, B_1\,, \\
    \psi\,&=\, 0 && \text{on }\, \partial B_1\,,\\
   \end{aligned}
 \right.
\end{equation}
with estimate
\begin{equation}\label{PSI2}
 \|\psi\|_{L^{\infty}(B_1)}+ \|\nabla \psi\|_{L^2(B_1)}\,\leq\, C\varepsilon_0\,.
\end{equation}
Then Proposition \ref{GradEst} holds.
\end{lemma}
\begin{proof}
The proof is essentially taken from \cite{CM1}. Substituting \eqref{PSI1} into the left-hand side of \eqref{Convexity4} yields (using also that $u(t_1) = u(t_2) = \chi$ on $\partial B_1$)
\begin{align}\label{PSI3}
\int_{B_1}|u(t_1) - u(t_2)|^2&|\nabla u(t_2)|^2 = \int_{B_1}
|u(t_1) - u(t_2)|^2\Delta\psi\leq\int_{B_1} |\nabla |u(t_1) - u(t_2)|^2||\nabla
\psi| \notag\\
\leq\,&2\left(\int_{B_1}|\nabla u(t_1)-\nabla u(t_2)|^2\right)^{1/2}\left(\int_{B_1}\,|u(t_1) - u(t_2)|^2|\nabla
\psi|^2\right)^{1/2},
\end{align}
where we have applied Stokes' theorem to $\text{div}(|u(t_1)-u(t_2)|^2 \nabla\psi)$ and
used Cauchy-Schwarz inequality. Now applying Stokes' theorem to $\text{div}(|u(t_1)-u(t_2)|^2\psi \nabla\psi)$ and using that $\Delta\psi\,\geq\,0$ and \eqref{PSI3}, we have
\begin{align}\label{PSI4}
\int_{B_1}|u(t_1)-&u(t_2)|^2|\nabla \psi|^2 \leq\,\int_{B_1}
|\psi|(|u(t_1)-u(t_2)|^2\Delta\psi+|\nabla |u(t_1)-u(t_2)|^2||\nabla
\psi|)\notag\\
\leq &\,4\|\psi\|_{L^{\infty}}\left(\int_{B_1}|\nabla u(t_1)-\nabla u(t_2)|^2\right)^{1/2}\left(\int_{B_1}|u(t_1)-u(t_2)|^2|\nabla \psi|^2\right)^{1/2},
\end{align}
and so that
\begin{equation}\label{PSI5}
\left(\int_{B_1} |u(t_1)-u(t_2)|^2|\nabla \psi|^2\right)^{1/2}\leq
4\|\psi\|_{L^{\infty}}\left(\int_{B_1}|\nabla u(t_1)-\nabla u(t_2)|^2\right)^{1/2}.
\end{equation}
Finally, substituting \eqref{PSI5} back into \eqref{PSI3} and combining with \eqref{PSI2} (and choosing $\varepsilon_0$ sufficiently small) yield
\begin{align*}
\int_{B_1}|u(t_1) - u(t_2)|^2|\nabla u(t_2)|^2&\leq C \|\psi\|_{L^{\infty}}\int_{B_1}|\nabla u(t_1)-\nabla u(t_2)|^2\\
&\leq C\varepsilon_0\int_{B_1}|\nabla u(t_1)-\nabla u(t_2)|^2\,,
\end{align*}
which is just \eqref{Convexity4}.
\end{proof}

Similarly we can show the following lemma which states, under some extra condition, $\int_{B_1}|u_t(t)|^2$ is non-increasing along the harmonic map heat flow after some $T_0>0$ and Proposition \ref{utEst} could be validated in this case.
\begin{lemma}\label{utEst1}
Let $u(x,t)$ be as in Theorem \ref{MainThm}. For any $t_2 > t_1 \geq T_0 \geq 1$, suppose that for any $t_0\in [t_1, t_2]$ we can solve the following Dirichlet problem for $\psi \in W^{1,2}_0 \cap L^\infty(B_1)$:
\begin{equation}\label{PSI1-1-1}
\left\{
   \begin{aligned}
     \Delta \psi\,& = \,|\nabla u(t_0)|^2 && \text{in }\, B_1\,, \\
    \psi\,&=\, 0 && \text{on }\, \partial B_1\,,\\
   \end{aligned}
 \right.
\end{equation}
with estimate
\begin{equation}\label{PSI2-2-2}
 \|\psi\|_{L^{\infty}(B_1)}+ \|\nabla \psi\|_{L^2(B_1)}\,\leq\, C\varepsilon_0\,.
\end{equation}
Then we have
\begin{equation}\label{PSI5-5-5}
\int_{B_1} |u_t(t_2)|^2 \,\leq\,  \int_{B_1} |u_t(t_1)|^2\,.
\end{equation}
In particular, Proposition \ref{utEst} holds if \eqref{PSI1-1-1} and \eqref{PSI2-2-2} are valid for any $t_0\in[t_1,t_2]$ and any $t_2 > t_1 \geq T_0 \geq 1$.
\end{lemma}
\begin{proof}
Differentiate the flow equation \eqref{HHF} with respect to $t$, multiply with $u_t$, and integrate over $B_1\times [t_1,t_2]$, we have (e.g. treating $u_t$ as a difference quotient: $u_t(\cdot,t) = \lim_{h\to 0^+}(u(\cdot, t+h)-u(\cdot,t))/h$ which is zero on $\partial B_1$ for all $t\geq1$)
\begin{align}\label{PSI3-3-3}
\frac{1}{2}\int_{t_1}^{t_2}\int_{B_1} \partial_t |u_t|^2 + &\int_{t_1}^{t_2}\int_{B_1} |\nabla u_t|^2 \leq C \int_{t_1}^{t_2}\int_{B_1}|u_t|^2|\nabla u|^2 + |u_t||\nabla u||\nabla u_t|\notag\\
\leq &\, \frac{1}{2}\int_{t_1}^{t_2}\int_{B_1} |\nabla u_t|^2  + C \int_{t_1}^{t_2}\int_{B_1}|u_t|^2|\nabla u|^2\,.
\end{align}
Since \eqref{PSI1-1-1} and \eqref{PSI2-2-2} are valid for any $t_0\in [t_1, t_2]$, we can use the same argument as in the proof of Lemma \ref{GradEst1} to get estimate for
$$\int_{B_1}|u_t|^2|\nabla u|^2$$
at the time $t_0$ slice. Indeed, similar to \eqref{Convexity4} (i.e., replacing $u(t_1)- u(t_2)$ by $u_t(t_0)$), for any $t_0\in [t_1, t_2]$ we have
\begin{equation}\label{PSI4-4-4}
\int_{B_1}|u_t|^2|\nabla u|^2(t_0) \leq C \|\psi\|_{L^{\infty}}\int_{B_1}|\nabla u_t(t_0)|^2\leq C\varepsilon_0\int_{B_1}|\nabla u_t(t_0)|^2\,.
\end{equation}
Inserting \eqref{PSI4-4-4} back into \eqref{PSI3-3-3} (for any $t_0\in [t_1, t_2]$) we see that the right-hand side of \eqref{PSI3-3-3} can be absorbed into the the left-hand side if we choose $\varepsilon_0$ sufficiently small. This implies that we have \eqref{PSI5-5-5} for any such $t_2 > t_1 \geq T_0$.

If \eqref{PSI1-1-1} and \eqref{PSI2-2-2} are valid for any $t_0\in[t_1,t_2]$ and any $t_2 > t_1 \geq T_0 \geq 1$, then, in view of \eqref{PSI5-5-5}, estimating by the mean value of $|u_t|^2$ over $B_1\times[t_1,t_2]$ gives Proposition \ref{utEst}.
\end{proof}

Therefore, everything boils down to validating the assumptions in Lemmas \ref{GradEst1} and \ref{utEst1}, i.e., the existence of such functions $\psi$'s satisfying \eqref{PSI1}, \eqref{PSI2} and \eqref{PSI1-1-1}, \eqref{PSI2-2-2}, respectively, for any $t_0\geq T_0$ for some $T_0\geq 1$. We point out that, a prior we only know that the energy density $|\nabla u(t)|^2$ lies in $L^1(B_1)$ with global estimate $\||\nabla u(t)|^2\|_{L^1(B_1)} \leq \varepsilon_0$ for any fixed $t$. But $L^1$ is the borderline case in which the standard $L^p$-theory for the Dirichlet problem \eqref{PSI1} with estimate \eqref{PSI2} fails!

However, the following regularity theorem for boundary value problems in the local Hardy space $h^1(B_1)$ sheds a new light to validate the assumptions in Lemmas \ref{GradEst1} and \ref{utEst1}. Here the local Hardy space $h^1(B_1)$ is a strict subspace of $L^1(B_1)$ and we will recall its definition in Definition \ref{10003} below.
\begin{theorem}\label{CDSTHM}(cf. \cite[Theorem 1.100]{Sem} and \cite[Theorem 5.1]{CKS})
Let $f\in h^1(B_1)$ and assume that $f\geq 0$ a.e. in $B_1$. Then there exists a function $\psi \in L^\infty\cap W_0^{1,2}(B_1)$ solving the Dirichlet problem
 \begin{equation}
  \left\{
   \begin{aligned}
     \Delta \psi\,&=\, f && \text{in }\, B_1\,, \\
    \psi\,&=\, 0 && \text{on }\, \partial B_1\,.\\
   \end{aligned}
 \right.
   \end{equation}
Moreover, there exists a constant $C>0$ such that
\begin{equation}
\|\psi\|_{L^{\infty}(B_1)}  + \|\nabla \psi\|_{L^2(B_1)} \,\leq\, C \,\|f\|_{h^1(B_1)}\,.
\end{equation}
\end{theorem}
\begin{proof}
An elementary proof can be found in \cite[Theorem A.4]{LL}, which follows along the lines of a corresponding result in \cite{Sem}.
\end{proof}
\begin{remark}
This theorem can be thought of as a generalization of the Wente's lemma \ref{Wen} in Section \ref{Sec2}, cf. \cite{Mu}. For a more general version of this theorem we refer to Chang, Krantz and Stein's work in \cite{CKS}.
\end{remark}

\begin{definition}\label{10003}(\cite{Mi}) Choose a Schwartz function $\phi \in C^{\infty}_0(B_1)$ such that $\int_{B_1} \phi \,dx =1$ and let $\phi_t(x) = t^{-2}\phi\left(\frac{x}{t}\right)$. For a measurable function $f$ defined in $B_1$ we say that $f$ lies in the local Hardy space $h^1(B_1)$ if the radial maximal function of $f$
\begin{equation}\label{Maxima}
f^{\ast}(x) = \sup_{0<t< 1-|x|}\left| \int_{B_t(x)}\frac{1}{t^2}\phi\left(\frac{x-y}{t}\right) f(y)dy\right|(x)= \sup_{0<t< 1-|x|}\left| \phi_t \ast f\right|(x)
\end{equation}
belongs to $L^1(B_1)$ and we define
\begin{equation}
\|f\|_{h^1(B_1)} = \|f^{\ast}(x)\|_{L^1(B_1)}\,.
\end{equation}
It follows immediately that $h^1(B_1)$ is a strict subspace of $L^1(B_1)$ and $\|f\|_{L^1(B_1)}\leq \|f\|_{h^1(B_1)}$. It is also clear that if $f\in L^p(B_1)$ for some $p>1$ then $\|f\|_{h^1(B_1)}\leq C\|f\|_{L^p(B_1)}$.
\end{definition}

We shall remark that the local Hardy spaces $h^1$ (or the global version $\mathcal{H}^1$) act as replacements to $L^1$ in Calderon-Zygmund estimates. Therefore by Theorem \ref{CDSTHM}, if we can somehow manage to obtain a ``slightly'' improved global estimate for $|\nabla u|^2$ from $L^1(B_1)$ to $h^1(B_1)$ for all $t_0\geq T_0$, it will be sufficient to validate the assumptions in Lemmas \ref{GradEst1} and \ref{utEst1}. As mentioned above, the subtlety is that, without imposing any regularity information on the boundary data $\chi$, global estimates are very difficult to obtain.

The rest of the paper is devoted to validating the assumptions in Lemmas \ref{GradEst1} and \ref{utEst1}. Namely, in view of Theorem \ref{CDSTHM}, it suffices to show there exists $T_0>0$ such that
\begin{equation}\label{BoilDownTo}
\||\nabla u(t_0)|^2\|_{h^1(B_1)}\leq C\varepsilon_0 \quad \text{for any } t_0\geq T_0.
\end{equation}
The point here is that no pointwise estimate on $\nabla u$ such as \eqref{TEMP50} is needed, and instead, a (weaker) improved global integral estimate \eqref{BoilDownTo} will be sufficient and turns out to be the key to the proof of the main Theorem \ref{MainThm}.

\section{Analysis of harmonic map heat flow using Rivi\`ere's gauge}\label{Sec2}
Regarding the regularity of weakly harmonic maps from surfaces, H\'elein (see e.g. \cite{He1}) proved the interior regularity with the help of the so called Coulomb or moving frame, and Qing \cite{Q} showed the continuity up to the boundary in the case of continuous boundary data based on H\'elein's technique. More recently, Rivi\`ere \cite{Riv1} succeeded in writing the $2$-dimensional conformally invariant non-linear system of elliptic PDE's (which includes the weakly harmonic map equation \eqref{HM}) in the following form:
\begin{equation}\label{CIPDE1}
-\Delta u^i\,=\, \Omega^i_{j}\cdot \nabla u^j \quad i=1,2,...,n \quad\text{or}\quad -\Delta u\,=\, \Omega\cdot \nabla u\,,
\end{equation}
with $\Omega=(\Omega^i_j)_{1\leq i,j\leq n}\in L^2(B_1, so(n)\otimes\wedge^1\mathbb{R}^2)$ and $\Omega^i_{j}=-\Omega^j_{i}$ (antisymmetry). Here and throughout the paper, the Einstein summation convention is used. We refer to the system of equations \eqref{CIPDE1} as Rivi\`ere's equation. In particular, this special form of the nonlinearity enabled Rivi\`ere to obtain a conservation law for this system of PDE's (see \eqref{A-B-3} below), which is accomplished via a technique that we call Rivi\`ere's gauge decomposition. More precisely, following the strategy of Uhlenbeck in \cite{Uh}, Rivi\`ere \cite{Riv1} used the algebraic feature of $\Omega$, namely $\Omega$ being antisymmetric, to construct $\xi\in W^{1,2}_{0}(B_1, so(n))$ and a gauge transformation matrix $P \in W^{1,2}\cap L^\infty(B_1, SO(n))$ (which pointwise almost everywhere is an orthogonal matrix in $\mathbb{R}^{n\times n}$) satisfying some good properties.
\begin{theorem}\label{PCloseTOID}(\cite[Lemma A.3]{Riv1}) There exist $\varepsilon>0$ and $C>0$ such that for every $\Omega$ in $L^2(B_1, so(n)\otimes\wedge^1\mathbb{R}^2)$ satisfying
$$
\int_{B_1}|\Omega|^2\,\leq\, \varepsilon\,,
$$
there exist $\xi\in W_0^{1,2}(B_1, so(n))$ and $P\in W^{1,2}(B_1, SO(n))$ such that
\begin{equation}\label{P-1}
\nabla^{\perp}\xi \,=\, P^{T}\nabla P + P^T\Omega P \text{ in} \,\,B_1 \quad\text{with}\quad \xi \,=\,0\text{ on} \,\,\partial B_1,
\end{equation}
and
\begin{equation}\label{P-2}
\|\nabla \xi\|_{L^2(B_1)}+ \|\nabla P\|_{L^2(B_1)} \,\leq\, C\|\Omega\|_{L^2(B_1)}\,.
\end{equation}
Here the superscript $T$ denotes the transpose of a matrix.
\end{theorem}
\begin{remark}Multiplying both sides of equation \eqref{P-1} by $P$ from the left gives that (with indices and $1\leq m, z\leq n$)
\begin{equation} \label{P-4}
\nabla P^i_j\,=\,P^i_m\nabla^{\perp} \xi^m_j - \Omega^i_z \,P^z_j,\quad  1\leq i,j \leq n\,.
\end{equation}
\end{remark}
\begin{remark}\label{PMATRIXX}
Besides Uhlenbeck's method there is another way to construct the gauge tranformation matrix $P$, namely one can minimize the energy functional
\begin{equation}
E(R) \,=\, \int_{B_1} \left|R^T\nabla R + R^T\Omega R\right|^2
\end{equation}
among all $R\in W^{1,2}(B_1,SO(n))$, see e.g. \cite{Cho} and \cite{Sc}.
\end{remark}

Another key result from Rivi\`ere's work is the following theorem, which was proved based on Theorem \ref{PCloseTOID}.

\begin{theorem}\label{PCloseTOIDllz}(\cite[Theorem I.4]{Riv1}) There exist $\varepsilon>0$ and $C>0$ such that for every $\Omega$ in $L^2(B_1, so(n)\otimes\wedge^1\mathbb{R}^2)$ satisfying
$$
\int_{B_1}|\Omega|^2\,\leq\, \varepsilon\,,
$$
there exist $\widehat{A} \in W^{1,2}\,\cap\, C^{0}(B_1, Gl_n(\mathbb{R}))$, $A = (\widehat{A}+ Id)\,P^{T} \in L^{\infty}\,\cap\, W^{1,2}(B_1, Gl_n(\mathbb{R}))$ and $B\in W_0^{1,2}(B_1, M_n(\mathbb{R}))$ such that
\begin{equation}\label{A-B-1}
\nabla A - A \Omega\,=\, \nabla^{\perp} B
\end{equation}
and
\begin{equation}\label{A-B-2}
\|\widehat{A}\|_{W^{1,2}(B_1)} + \|\widehat{A}\|_{L^\infty(B_1)} + \|B\|_{W^{1,2}(B_1)}\,\leq\, C \|\Omega\|_{L^2(B_1)}\,.
\end{equation}
\end{theorem}
\begin{remark}
Combining \eqref{A-B-1} with \eqref{CIPDE1} one obtains the conservation law (in the distribution sense) for Rivi\`ere's equation \eqref{CIPDE1}:
 \begin{equation}\label{A-B-3}
    \text{div }(A\nabla u+B\nabla^{\perp}u)\,=\,0\,.
   \end{equation}
\end{remark}

Equation \eqref{CIPDE1}, first considered in such generality by Rivi\`ere in \cite{Riv1}, generalizes a number of interesting equations appearing naturally in geometry, including the harmonic map equation \eqref{HM}), the $H$-surface equation and, more generally, the Euler-Lagrange equation of any conformally invariant elliptic Lagrangian which is quadratic in the gradient. We remark that the harmonic map equation \eqref{HM} can be written in the form of \eqref{CIPDE1} if we set
\begin{equation}\label{SFF}
\Omega:=(\Omega_j^i)_{1\leq i,j\leq n}\quad \text{where } \Omega^i_j:=\,[A^i(u)_{j,l}-A^j(u)_{i,l}]\nabla u^l\,.
\end{equation}
A central issue is the regularity of the weak solution $u$ to this system of equations \eqref{CIPDE1}. Based on the conservation law \eqref{A-B-3}, Rivi\`ere proved the (interior) continuity of any $W^{1,2}$ weak solution $u$ to \eqref{CIPDE1}. This also resolved two conjectures by Heinz and Hildebrandt respectively, see \cite{Riv1}. We point out that the harmonic map heat flow \eqref{HHF} on $B_1$ can be written in the form
\begin{equation}\label{HHFNEW}
u_t-\Delta u =  \Omega\cdot \nabla u \quad \text{on } B_1\times (0,T)\,,
\end{equation}
where $\Omega$ is as in \eqref{SFF}.

The deep reason for Rivi\`ere's argument to work is that once the conservation law \eqref{A-B-3} is established, then equation \eqref{CIPDE1} can be rewritten in the form
$$
\text{div}(A\nabla u)\,=\,\nabla^\perp B \cdot\nabla u\,,
$$
The right hand side of this new equation lies in the Hardy space $\mathcal{H}^1$ by a result of Coifman, Lions, Meyer and Semmes \cite{CLMS}. Moreover, using a Hodge decomposition argument, one can show that $u$ lies locally in $W^{2,1}$ which embeds into $C^0$ in two dimensions, cf. the proof of Theorem \ref{Regularity} below. The key to this fact is a special ``compensation phenomena'' for Jacobian determinants, which was first observed by Wente in \cite{W}. We will refer it to the following Wente's lemma, for which an elementary proof can be found in \cite{BC1} and \cite[Theorem 3.1.2]{He1}. It will also be the key ingredient of our proof in this paper.

\begin{lemma}[Wente's lemma \cite{W}] \label{Wen} If $a,b\in W^{1,2}(B_1,\mathbb{R})$ and $w$ be the solution of
\begin{equation}
\left\{
\begin{aligned}
\Delta w\, &=\,\frac{\partial a}{\partial y}\frac{\partial b}{\partial x} - \frac{\partial a}{\partial x}\frac{\partial b}{\partial y} \,=\, \nabla a \cdot \nabla^{\perp} b && \text{in } \, B_1\,,\\
w\, &=\,0 \quad \text{or} \quad \frac{\partial w}{\partial \nu}\,=\,0 && \text{on } \, \partial B_1\,.
\end{aligned}
\right.
\end{equation}
Then $w\in C^0\cap W^{1,2}(B_1,\mathbb{R})$ and the following estimate holds
\begin{equation}
\|w\|_{L^\infty(B_1)}+ \|\nabla w\|_{L^2(B_1)}\,\leq\,C\|\nabla a\|_{L^2(B_1)}\|\nabla b\|_{L^2(B_1)}\,,
\end{equation}
where we choose $\int_{B_1} w = 0 $ for the Neumann boundary data.
\end{lemma}

Now let $u(x,t) \in W^{1,2}(B_1\times[0,\infty),\mathcal{N})$ be a global weak solution to the harmonic map heat flow \eqref{HHF} with $E(u_0) < \varepsilon_0$, $E(u(\cdot, t))$ is non-increasing in $t$ and $u(\cdot, t)|_{\partial B_1} = \chi$ for all $t\geq 0$ as in Theorem \ref{MainThm}. First note that for a.e. $t_0\in (0,\infty)$ we have $u_t(t_0)\in L^2(B_1)$. Then for any fixed $t_0$ such that $u_t(t_0) \in L^2(B_1)$, as in \eqref{SFF}, we have
$$
\Omega(t_0) =(\Omega_j^i(t_0))_{1\leq i,j\leq n}\,\, \text{ where } \,\, \Omega^i_j(t_0)=\,[A^i(u(t_0))_{j,l}-A^j(u(t_0))_{i,l}]\nabla u^l(t_0).
$$
We will abbreviate $\Omega(t_0)$ to $A(u(t_0))\nabla u(t_0)$. Moreover,
\begin{equation}
\int_{B_1}|\Omega(t_0)|^2 \,\leq\, C E(u(t_0)) \,\leq\, C\varepsilon_0\,.
\end{equation}
Therefore Rivi\`ere's Theorems \ref{PCloseTOID} and \ref{PCloseTOIDllz} on the existence of gauge apply to this time $t_0$ slice, and we find the existence of matrices $P(t_0)\in W^{1,2}(B_1, SO(n))$,
$$A(t_0)  = (\widehat{A}(t_0) + Id)\,P^{T}(t_0) \,\,\in\,\, L^{\infty}\,\cap\, W^{1,2}(B_1, Gl_n(\mathbb{R}))$$
and
$$B(t_0)\,\,\in\,\, W_0^{1,2}(B_1, M_n(\mathbb{R}))$$ such that
\begin{equation}\label{A-B-1-1}
\nabla A(t_0) - A(t_0) \Omega(t_0)\,=\, \nabla^{\perp} B(t_0)
\end{equation}
with the corresponding estimates \eqref{P-2} and \eqref{A-B-2}.

Combining \eqref{A-B-1-1} with the harmonic map heat flow equation \eqref{HHFNEW} yields (omitting the index $t_0$)
\begin{align}\label{A-B-3-3}
\text{div }(A \nabla u + B \nabla^{\bot} u)\,&=\,\nabla A \cdot \nabla u + A \Delta u + \nabla B \cdot \nabla^{\bot} u \notag\\
             &= \nabla A\cdot \nabla u +A(-\Omega\cdot \nabla u +u_t)+\nabla B\cdot \nabla^{\bot} u\notag\\
             &= \nabla A\cdot \nabla u + (\nabla^{\bot} B-\nabla A) \cdot \nabla u +Au_t +\nabla B\cdot \nabla^{\bot} u\\
             &= Au_t \,.\notag
 \end{align}
 We refer to \eqref{A-B-3-3} as an ``almost'' conservation law. By the results of \cite{CLMS} and the standard $L^p$ theory, \eqref{A-B-3-3} readily implies that $u(t_0) \in C^0(B_1, \mathbb{R}^n)$. In fact, we have the following $\varepsilon$-regularity theorem.
 \begin{theorem}\label{Regularity}
There exist $\varepsilon_0 >0$ depending only on $\mathcal{N}$ such that if $u \in W^{1,2}(B_1\times[0,\infty),\mathcal{N})$ is a global weak solution to the harmonic map heat flow \eqref{HHF} with $E(u_0) < \varepsilon_0$, $E(u(\cdot, t))$ is non-increasing in $t$ and $u(\cdot, t)|_{\partial B_1} = \chi$ for all $t\geq 0$, then $u \in C^\infty(B_1\times [1,\infty), \mathcal{N})$.
\end{theorem}
\begin{proof}
For any fixed $t_0$ such that $u_t(t_0) \in L^2(B_1)$, by Hodge decomposition (see e.g. \cite[Corollary 10.5.1]{IM}), there exist $D(t_0), E(t_0) \in W^{1,2}(B_1, \mathbb{R}^n)$ such that (omitting the index $t_0$)
 \begin{equation}\label{Regularity220101}
 A\nabla u = \nabla D + \nabla^\perp E\,.
 \end{equation}
Note that \eqref{A-B-3-3} implies
 \begin{equation}\label{Regularity22}
  \left\{
   \begin{aligned}
    \text{div }( A\nabla u )  &= -\nabla B \cdot\nabla^\perp u + Au_t \\
    \text{curl }( A\nabla u ) &= \nabla^\perp A \cdot\nabla u\,.
   \end{aligned}
 \right.
   \end{equation}
Combining \eqref{Regularity220101} and \eqref{Regularity22} we have
 \begin{equation}
  \left\{
   \begin{aligned}
    \Delta D &= -\nabla B \cdot\nabla^\perp u + Au_t \\
    \Delta E &= \nabla^\perp A \cdot\nabla u\,.
   \end{aligned}
 \right.
   \end{equation}
Then by the results of \cite{CLMS} and via an extension argument, using the fact that $Au_t(t_0)\in L^2(B_1)$, we get $A\nabla u(t_0) \in W^{1,1}_{loc}(B_1)$.  Therefore $u(t_0)\in W^{2,1}_{loc}(B_1)$ which embeds into $C^0(B_1)$.

Indeed, $A\nabla u(t_0) \in W^{1,1}_{loc}(B_1)$ implies immediately that
\begin{equation}\label{TEMP4}\Omega(t_0) = A(u(t_0))\nabla u(t_0)\,\,\in\,\, W^{1,1}_{loc}(B_1)\,.\end{equation} Then by \cite[Theorem III.4]{Riv4}, the flow equation \eqref{HHFNEW}, \eqref{TEMP4} and the fact that $u_t(t_0)\in L^2(B_1)$ yield that $\nabla u(t_0)\in L^p_{loc}(B_1)$ for some $p>2$. Note that this is valid for a.e. $t_0 \in (0,\infty)$. Then via a standard bootstrapping argument we have $\nabla u\in L^q_{loc}(B_1\times[1,T])$ for all $q>1$ and any $T>1$, see e.g. \cite{Li}, and all higher order interior regularity follows.
\end{proof}
\noindent Again, we see that the ``compensation phenomenon'' enjoyed by the special Jacobian structure (see Wente's lemma \ref{Wen}) has played an important role here, and these Wente type estimates have many interesting applications as observed already in e.g. \cite{W,BC1, BC2, Ta2, CLMS,He1,Riv1,Riv2,Riv3} and our recent work \cite{LL}.

\section{Improved estimates on the matrices $B$ and $P$}\label{Sec3}
Our main observation in this section are two hidden Jacobian structures for $\Delta B$ and $\Delta P$, which are valid only for $\Omega$ being in some special form, i.e., in our case we have $\Omega = A(u)\nabla u$\,. This observation allows us to gain improved global estimate on the matrix $B$ and improved local estimate on the matrix $P$.

\subsection{Improved global estimate on the matrix $B$}
We will first show an improved global estimate on the matrix $B$.
\begin{proposition}\label{XI}
Let $u(x,t)$ be as in Theorem \ref{MainThm}, then for any $t_0 \in [1,\infty)$ we have
\begin{equation}\label{A-B-4-4}
\|B(t_0)\|_{L^\infty(B_1)}\,\leq\, C \int_{B_1}|\nabla u(t_0)|^2\,\leq\,C \varepsilon_0\,.\end{equation}
\end{proposition}
\begin{proof}
We recall that $\Omega$ is given by $A(u)\nabla u$ as in \eqref{SFF} and therefore $\|\Omega(t_0)\|^2_{L^2(B_1)} \leq C\varepsilon_0$ for all $t_0\geq 1$. Now let $\varepsilon_0$ be so small that Theorems \ref{PCloseTOID} and \ref{PCloseTOIDllz} apply. Taking the curl on both sides of equation \eqref{A-B-1-1} yields
\begin{equation}\label{A-B-5-5}
\Delta B(t_0)\,=\, - \text{curl}(A(t_0) A(u(t_0))\nabla u (t_0)) \,=\,- \nabla u (t_0) \cdot \nabla^{\perp}(A(t_0) A(u(t_0)))\,.
\end{equation}
Combining the Jacobian structure of the right-hand side of \eqref{A-B-5-5} with the zero boundary condition of $B$ and estimates \eqref{P-2} and \eqref{A-B-2}, the Wente's lemma \ref{Wen} gives \eqref{A-B-4-4}. Here we have also used $E(u(t))<\varepsilon_0$ for all $t\geq 0$ , the smoothness and compactness of the target manifold $\mathcal{N}$.
\end{proof}

\subsection{Improved local estimate on the matrix $P$}\label{MatrixU}
We next show that $\Delta P$ also has a special Jacobian structure.
\begin{lemma}\label{P2}
Let $u(x,t)$ be as in Theorem \ref{MainThm}, then for any $t_0 \in [1,\infty)$ such that $u_t(\cdot, t_0)\in L^2(B_1)$, there exist $\xi(t_0)\in W_0^{1,2}(B_1, so(n))$, $\eta(t_0)\in W^{1,2}(B_1,\mathbb{R}^n)$, $\zeta(t_0)\in W_0^{2,2}(B_1,\mathbb{R}^n)$ and $Q_k(t_0), R_k(t_0) \in W^{1,2}\cap L^{\infty}(B_1, Gl_n(\mathbb{R}))$, $k=1,...,n$ with
\begin{align*}
\|\nabla \xi(t_0)\|_{L^{2}(B_1)} &+ \|\nabla \eta(t_0)\|_{L^{2}(B_1)} +  \|\nabla \zeta(t_0)\|_{L^{2}(B_1)} \\
&+\sum_k \left(\|\nabla Q_k(t_0)\|_{L^{2}(B_1)} + \|\nabla R_k(t_0)\|_{L^{2}(B_1)}\right) \leq C\sqrt{\varepsilon_0}
\end{align*}
and
\begin{equation}\label{TEMP}
 \| \zeta(t_0)\|_{W^{2,2}(B_1)} \leq C \|u_t(t_0)\|_{L^{2}(B_1)}
\end{equation}
such that
\begin{align}\label{P4}
\Delta P(t_0) = &\nabla P(t_0) \cdot\nabla^{\perp} \xi(t_0) + \nabla Q_k(t_0) \cdot \nabla^{\perp} \eta^k(t_0) \notag\\
&+ \nabla R_k(t_0) \cdot \nabla^{\perp} u^k(t_0) + \text{div } ( Q_k(t_0) \nabla \zeta^k(t_0))\,.
\end{align}
\end{lemma}

\begin{proof}
We omit the index $t_0$ in the proof. Hodge decomposition and the estimates for the $L^\infty$-norms of $A$ and $B$ imply the existence of $\eta\in W^{1,2}(B_1,\mathbb{R}^n)$ and $\zeta\in W_0^{1,2}(B_1,\mathbb{R}^n)$ such that
\begin{equation}\label{A-B-6-6}
\nabla^{\perp} \eta + \nabla\zeta\,=\, A\nabla u + B\nabla^{\perp} u
\end{equation}
with
\begin{equation}\label{A-B-7-7}
\|\nabla \eta\|_{L^2(B_1)} + \|\nabla \zeta\|_{L^2(B_1)}\leq C \|\nabla u\|_{L^2(B_1)}\leq C\sqrt{\varepsilon_0}\,.
\end{equation}
Moreover, by the ``almost'' conservation law \eqref{A-B-3-3} we have
$$\Delta \zeta\,=\, Au_t \,\,\in\, \,L^2(B_1)\quad\text{and}\quad \zeta|_{\partial B_1} = 0\,, $$
which gives \eqref{TEMP} by the standard $L^p$-theory. Multiplying both sides of equation \eqref{A-B-6-6} by $A^{-1}$ from the left gives (with indices)
\begin{equation}\label{A-B-8-8}
\nabla u^{l}\,=\, (A^{-1})^{l}_{k}\nabla^{\perp} \eta^k - (A^{-1} B)^{l}_{k} \nabla^{\perp} u^k + (A^{-1})^{l}_{k}\nabla\zeta^k\,,\quad l=1,2,...,n\,.
\end{equation}
Taking divergence on both sides of equation \eqref{P-4} yields
\begin{equation} \label{A-B-9-9}
\Delta P^i_{j} \,=\, \nabla P^i_m\cdot\nabla^{\perp} \xi^m_j - \text{div }(\Omega^i_z \,P^z_j)\,,\quad 1\leq i,j\leq n\,.
\end{equation}

Since $\Omega^i_z = [A^i(u)_{z,l} - A^z(u)_{i,l}]\nabla u^l$, combining \eqref{A-B-8-8} and \eqref{A-B-9-9} gives
\begin{align}
\Delta P^i_{j} \,= &\,\nabla P^i_m\cdot\nabla^{\perp} \xi^m_j- \text{div }\big[(A^i(u)_{z,l} - A^z(u)_{i,l})\,\big((A^{-1})^{l}_{k}\nabla^{\perp} \eta^k \notag\\
&\,\quad \quad \quad \quad \quad \quad \quad \quad \quad \quad - (A^{-1} B)^{l}_{k} \nabla^{\perp} u^k + (A^{-1})^{l}_{k}\nabla\zeta^k\big)\, P^z_j\big]\notag\\
=& \,\nabla P^i_m\cdot\nabla^{\perp} \xi^m_j - \nabla \left[(A^i(u)_{z,l} - A^z(u)_{i,l}) P^z_j (A^{-1})^l_k\right]\cdot \nabla^{\perp} \eta^k \\
&\,+ \nabla \left[(A^i(u)_{z,l} - A^z(u)_{i,l}) P^z_j (A^{-1} B)^l_k\right]\cdot \nabla^{\perp} u^k \notag\\
& \,- \text{div }\left[\left( (A^i(u)_{z,l} - A^z(u)_{i,l}) P^z_j (A^{-1})^l_k\right)\nabla \zeta^k\right] \,.\notag
\end{align}
Defining $(Q_k)^i_j = -(A^i(u)_{z,l} - A^z(u)_{i,l}) P^z_j (A^{-1})^l_k$ and $(R_k)^i_j=(A^i(u)_{z,l}- A^z(u)_{i,l})\cdot$ $P^z_j (A^{-1} B)^l_k\,,$ where $1\leq k, i,j \leq n$, completes the proof.
\end{proof}

Next we prove a local estimate on the oscillation of the matrix $P$ based on Lemma \ref{P2}. A key observation here is that whether a function is in the local Hardy space $h^1(B_1)$ essentially depends on its local behavior (see Definition \ref{10003}). This local oscillation estimate on $P$ provides important information that we need to control the local behavior of $|\nabla u|^2$. This point will become apparent in Section \ref{Sec4}. As we shall see, the Jacobian structure of $\Delta P$ enters in a crucial way.

\begin{lemma}\label{PCloseTOIDllz1}
Let $u(x,t)$ be as in Theorem \ref{MainThm}, then for any $t_0 \in [1,\infty)$ such that $u_t(\cdot, t_0)\in L^2(B_1)$, any $x\in B_1$, any $r>0$ such that $B_{2r}(x) \subset B_1$ and any $y\in B_{r}(x)$ we have
\begin{equation}\label{U-1}
|P(y,t_0) - P(x,t_0)|\,\leq\, C \left(\sqrt{\varepsilon_0}+\|u_t(t_0)\|_{L^2(B_1)}\right)\,.
\end{equation}
\end{lemma}
\begin{proof} We will omit the index $t_0$ in the proof.
Let $\widetilde{P} \in W^{1,2}(B_1, M_n(\mathbb{R}))$ be the weak solution of
$$
  \left\{
   \aligned
   \Delta \widetilde{P}\,=\,& \,\nabla P \cdot\nabla^{\perp} \xi + \nabla Q_k \cdot \nabla^{\perp} \eta^k + \nabla R_k \cdot \nabla^{\perp} u^k + \text{div } ( Q_k \nabla \zeta^k)  &&\text{in  } B_1 \,, \\
   \widetilde{P} \, = \, &0  &&\text{on  }\partial B_1\,,\\
   \endaligned
  \right.
$$
where $Q_k$, $R_k$, $\eta^k$ and $\zeta^k$'s are from Lemma \ref{P2}.

Then by Wente's lemma \ref{Wen} and the standard $L^p$-theory (and $W^{2,2}(B_1)\hookrightarrow C^0(B_1)$), we have $\widetilde{P}\in C^0(B_1, M_n(\mathbb{R}))$ and
\begin{equation}\label{100}
\|\widetilde{P}\|_{L^\infty(B_1)}+\|\nabla \widetilde{P}\|_{L^2(B_1)}\,\leq \, C \left(\varepsilon_0+\|u_t(t_0)\|_{L^2(B_1)}\right)\,.
\end{equation}
Since $\Delta (P - \widetilde{P}) = 0$ in $B_1$, we know that $V = P - \widetilde{P} \in C^{\infty}(B_1, M_n(\mathbb{R}))$ is harmonic. Now for any $x\in B_1$, any $r>0$ such that $B_{2r}(x) \subset B_1$ and any $y\in B_{r}(x)$ we have
\begin{align}\label{1000}
&|V(y) - V(x)|\,\leq\, Cr\|\nabla V\|_{L^{\infty}(B_r(x))} \leq\, Cr^{-1}\|\nabla V\|_{L^1(B_{2r}(x))}\,\notag\\
\leq \, & C\|\nabla V\|_{L^2(B_{2r}(x))}
\,\leq\, C \left(\|\nabla P\|_{L^2(B_{2r}(x))} + \|\nabla \widetilde{P}\|_{L^2(B_{2r}(x))}\right)\\
\,\leq\,&C \left(\sqrt{\varepsilon_0}+\|u_t(t_0)\|_{L^2(B_1)}\right) \,,\notag
\end{align}
where we have used the mean value property of $V$ and \eqref{100}, \eqref{P-2}. Combining \eqref{100} and \eqref{1000} yields that for any $x\in B_1$, any $r>0$ such that $B_{2r}(x) \subset B_1$ and any $y\in B_{r}(x)$ we have
\begin{equation}
|P(y, t_0)- P(x, t_0)|\,\leq\, C \left(\sqrt{\varepsilon_0}+\|u_t(t_0)\|_{L^2(B_1)}\right)\,,
\end{equation}
which gives the desired estimate \eqref{U-1}.
\end{proof}

\section{Validating \eqref{BoilDownTo} and proof of Theorem \ref{MainThm}}\label{Sec4}
With the results so far at our disposal, we are now in a position to validate \eqref{BoilDownTo}. As mention above, the local estimate on the oscillation of the transformation matrix $P$ in Lemma \ref{PCloseTOIDllz1} will be the key ingredient.
\begin{lemma}\label{10001} Let $u(x,t)$ be as in Theorem \ref{MainThm}, then for any $t_0 \in [1,\infty)$ such that $\|u_t(t_0)\|_{L^2(B_1)} < \sqrt{\varepsilon_0}$ we have
\begin{equation}
|\nabla u(t_0)|^2 \in h^1(B_1)
\end{equation}
with estimate
\begin{equation}
\||\nabla u(t_0)|^2\|_{h^1(B_1)}\leq C\varepsilon_0\,.
\end{equation}
\end{lemma}
\begin{remark}
Lemma \ref{10001} continues to hold if we replace $\Omega^i_j = [A^i(u)_{j,l}-A^j(u)_{i,l}]\nabla u^l$ in \eqref{SFF} by a more general $\Omega$ in the form $\Omega^i_j = \sum_{l=1}^n f^i_{jl} \nabla u^l + g^i_{jl} \nabla^{\perp} u^l$, see \cite{LL}. Moreover, the condition $\|u_t(t_0)\|_{L^2(B_1)} < \varepsilon_0$ can be replaced by that $\|u_t(t_0)\|_{L^p(B_1)}$ is sufficiently small for some $p>1$.
\end{remark}
\begin{proof}(of Lemma \ref{10001}) By the assumption $\|u_t(t_0)\|_{L^2(B_1)} < \sqrt{\varepsilon_0}$ and Lemma \ref{PCloseTOIDllz1}, for any $x\in B_1$, any $r>0$ such that $B_{2r}(x) \subset B_1$ and any $y\in B_{r}(x)$ we have
\begin{equation}\label{TEMP2}
|P(y,t_0)-P(x,t_0)| \,\leq \, C \sqrt{\varepsilon_0}\,.
\end{equation}
We will omit the index $t_0$ from now on. By Proposition \ref{XI}, Theorem \ref{PCloseTOID} and Theorem \ref{PCloseTOIDllz}, for any $x\in B_1$, any $r>0$ such that $B_{2r}(x) \subset B_1$ and any $y\in B_{r}(x)$ we have (choosing $\varepsilon_0$ sufficiently small)
\begin{align}
0&\leq \,\frac{1}{2}|\nabla u|^2(y) \leq \left(A\nabla u + B\nabla^{\perp} u\right) \cdot (P^T\nabla u)(y)\notag\\
&=  \left(A\nabla u + B\nabla^{\perp} u\right)\cdot \left[\left(P^T(x) + \left(P^T-P^T(x)\right)\right)\nabla u\right](y)\,,
\end{align}
and therefore by \eqref{A-B-6-6} and \eqref{TEMP2}
\begin{align}
&(\nabla^{\perp}\eta + \nabla\zeta)\cdot \left(P^T(x)\nabla u\right)(y) = \left(A\nabla u + B\nabla^{\perp} u\right) \cdot \left(P^T(x)\nabla u\right)(y)\notag\\
\geq &\,\frac{1}{2}|\nabla u|^2(y) - \left(A\nabla u + B\nabla^{\perp} u\right)\cdot \left[\left(P^T-P^T(x)\right)\nabla u\right](y)\geq \,\frac{1}{4}|\nabla u|^2(y)\,.\label{PPP}
\end{align}
Now we choose a function
\begin{equation}\label{funcphi}\phi \in C^{\infty}_0(B_1) \text{ with } \phi\geq 0, \,\text{spt}(\phi) \subseteq B_{\frac{1}{2}}, \,\phi =2 \text{ on } B_{\frac{3}{8}}, \text{ and } \int_{B_1} \phi \,dx =1\,.
\end{equation}
Moreover, we additionally assume that $\|\nabla \phi\|_{L^\infty(B_1)}\leq 100$. Using \eqref{TEMP} and \eqref{PPP}, one verifies directly that (by Definition \ref{10003})
\begin{align*}
\||\nabla u|^2\|_{h^1(B_1)}=& \,\int_{B_1} \sup_{0<t< 1-|x|} \phi_t \ast |\nabla u|^2dx\\
\leq& \,4\int_{B_1} \sup_{0<t< 1-|x|}\phi_t \ast \left((\nabla^{\perp}\eta + \nabla\zeta)\cdot (P^T(x)\nabla u)\right)dx\\
= & 4\int_{B_1} \sup_{0<t< 1-|x|} \phi_t \ast \left[(P^T(x))_{ij}\left(\nabla^{\perp}\eta^i \cdot \nabla u^j + \nabla\zeta^i \cdot \nabla u^j\right)\right]dx\\
\leq &\,C \,\sum_{i,j=1}^n\left(\|\nabla^\perp \eta^i \cdot \nabla u^j\|_{h^1(B_1)} + \|\nabla\zeta^i\|_{W^{1,2}(B_1)} \|\nabla u^j\|_{L^2(B_1)}\right)\\
\leq &C\, \|\nabla^{\perp}\eta\|_{L^2(B_1)}\|\nabla u\|_{L^2(B_1)} + C\sqrt{\varepsilon_0}\|u_t\|_{L^2(B_1)} \leq  C \varepsilon_0\,,
\end{align*}
where we have used the facts
\begin{enumerate}
\item $\nabla^{\perp}\eta^i \cdot \nabla u^j \in h^1(B_1)$ and $\|\nabla^{\perp}\eta^i \cdot \nabla u^j\|_{h^1(B_1)}\leq C\|\nabla \eta\|_{L^2(B_1)}\|\nabla u\|_{L^2(B_1)}$ for all $i,j=1,2,...,n$;\label{TEMP000}
\item $\|\nabla\zeta^i \cdot \nabla u^j\|_{L^p(B_1)}\leq C\|\nabla\zeta^i\|_{W^{1,2}(B_1)} \|\nabla u^j\|_{L^2(B_1)}$ for any $1<p<2$ and $\|f\|_{h^1(B_1)} \leq C\|f\|_{L^p(B_1)}$ for any $p>1$\,.
\end{enumerate}
 To see \eqref{TEMP000}, we first extend $\eta^i - \frac{1}{|B_1|}\int_{B_1}\eta^i$ and $u^j - \frac{1}{|B_1|}\int_{B_1}u^j$ from $B_1$ to $\mathbb{R}^2$ which yields the existence of $\tilde{\eta}^i, \tilde{u}^j \in W^{1,2}_c(\mathbb{R}^2)$ such that
\begin{equation}
\int_{\mathbb{R}^2}|\nabla \tilde{\eta}^i|^2 \leq C \int_{B_1}|\nabla \eta^i|^2\quad \text{and}\quad \int_{\mathbb{R}^2}|\nabla \tilde{u}^j|^2 \leq C \int_{B_1}|\nabla u^j|^2
\end{equation}
and
\begin{equation}\label{10006}
\nabla \tilde{\eta}^i = \nabla \eta^i \quad \text{and} \quad \nabla \tilde{u}^j = \nabla u^j \quad \text{a.e. in } B_1\,.
\end{equation}
Then by the results of \cite{CLMS} we know that
\begin{align}
\|\nabla^\perp \tilde{\eta}^i \cdot \nabla \tilde{u}^j\|_{\mathcal{H}^1(\mathbb{R}^2)}:&= \int_{\mathbb{R}^2} \sup_{\phi\in \mathcal{T}}\sup_{t>0}\left| \int_{B_t(x)}\frac{1}{t^2}\phi\left(\frac{x-y}{t}\right)\left(\nabla^\perp \tilde{\eta}^i \cdot \nabla \tilde{u}^j\right)(y)dy\right|dx\notag\\
& \leq C \|\nabla \tilde{\eta}^i\|_{L^2(\mathbb{R}^2)}\|\nabla \tilde{u}^j\|_{L^2(\mathbb{R}^2)} \leq C \|\nabla \eta\|_{L^2(B_1)}\|\nabla u\|_{L^2(B_1)}\,,\label{10007}
\end{align}
where $\mathcal{T} = \{\phi \in C^{\infty}(\mathbb{R}^2): \text{spt}(\phi)\subset B_1 \text{ and } \|\nabla \phi\|_{L^\infty} \leq 100\}$. By \eqref{10006}, \eqref{10007} and Definition \ref{10003}, it is clear that
\begin{align}\label{TEMP909}
\|\nabla^\perp \eta^i \cdot \nabla u^j\|_{h^1(B_1)} &= \|\nabla^\perp \tilde{\eta}^i \cdot \nabla \tilde{u}^j\|_{h^1(B_1)}\notag\\
&\leq \|\nabla^\perp \tilde{\eta}^i \cdot \nabla \tilde{u}^j\|_{\mathcal{H}^1(\mathbb{R}^2)} \leq C \|\nabla \eta\|_{L^2(B_1)}\|\nabla u\|_{L^2(B_1)}.
\end{align}
This completes the proof of the lemma.
\end{proof}

Now since $u(x,t)\in W^{1,2}\cap C^{\infty}(B_1\times[1,\infty),\mathcal{N})$ and the energy $E(u(\cdot,t))$ is non-increasing along the flow as shown in \eqref{EnergyDrease}, there exists $T_0 \geq 1$ such that
\begin{equation}\label{TEMP51}
\|u_t(T_0)\|_{L^2(B_1)} < \sqrt{\varepsilon_0}\,.
 \end{equation}
 Then by Lemma \ref{10001} we know that $|\nabla u(T_0)|^2 \in h^1(B_1)$ with estimate
\begin{equation}\label{TEMP_LLZ}
\||\nabla u(T_0)|^2\|_{h^1(B_1)}\leq C\varepsilon_0\,.
\end{equation}

Therefore, in view of Lemma \ref{10001}, in order to validate the global estimate \eqref{BoilDownTo} we are left to show
\begin{equation}
\|u_t(t_0)\|_{L^2(B_1)} < \sqrt{\varepsilon_0}\quad \text{for all } t_0 \geq T_0\,.
\end{equation}

We will next show this is indeed the case.

\begin{lemma}Let $u(x,t)$ be as in Theorem \ref{MainThm}, then there exists $T_0>0$ such that
\begin{equation}
\|u_t(t_0)\|_{L^2(B_1)} < \sqrt{\varepsilon_0}\quad \text{for all } t_0 \geq T_0\,.
\end{equation}
\begin{proof}
Let $T_0\geq 1$ be as in \eqref{TEMP51} so that we have $\|u_t(T_0)\|_{L^2(B_1)} < \sqrt{\varepsilon_0}$. Since $u(x,t)\in W^{1,2}\cap C^{\infty}(B_1\times[1,\infty),\mathcal{N})$ and by the continuity of $\int_{B_1}|u_t(t)|^2$ in $t$, there exists $\delta = \delta(T_0, \varepsilon_0)>0$ such that for any $t_0\in [T_0, T_0+\delta]$ we have
\begin{equation}\label{TEMP99}
\|u_t(t_0)\|_{L^2(B_1)} < 2\sqrt{\varepsilon_0}\,.
\end{equation}
Therefore by our previous arguments (especially Theorem \ref{CDSTHM}, Lemma \ref{10001} and \eqref{TEMP_LLZ} with $T_0$ replaced by $t_0$), Lemma \ref{utEst1} applies to any subinterval of $[T_0, T_0+\delta]$ and yields
\begin{equation}\label{utnon}
\int_{B_1} |u_t(t_2)|^2 \,\leq\,  \int_{B_1} |u_t(t_1)|^2 \quad \text{for any } t_1, t_2 \in [T_0, T_0+\delta]\,.
\end{equation}
This shows, instead of \eqref{TEMP99}, for any $t_0 \in [T_0, T_0+\delta]$ we have
\begin{equation}
\|u_t(t_0)\|_{L^2(B_1)} \leq \|u_t(T_0)\|_{L^2(B_1)}  < \sqrt{\varepsilon_0}\,.
\end{equation}
We can then continue and iterate this process beyond $T_0+\delta$ and we see that $\int_{B_1} |u_t(t)|^2$ is indeed non-increasing along the flow after $T_0$.
\end{proof}
\end{lemma}

This completes the validation of \eqref{BoilDownTo} and therefore the assumptions in Lemmas \ref{GradEst1} and \ref{utEst1} in view of Theorem \ref{CDSTHM}, finishing the proof of our main Theorem \ref{MainThm} as shown in Section \ref{Sec1}.

 \end{document}